\documentclass[a4paper,12pt]{amsart}
\usepackage{amssymb,amsfonts,amsmath}
\usepackage{ifthen}
\usepackage{graphicx}
\usepackage{float}
\newtheorem{theorem}{Theorem}[section]

\newtheorem{lemma}[theorem]{Lemma}

\theoremstyle{definition}

\newcounter{minutes}
\divide\time by 60
\newcounter{hours}
\multiply\time by 60
\addtocounter{minutes}{-\time}

\newcommand{\D}{{\mathbb D}}

\newcommand{\real}{{\operatorname{Re}\,}}

\newcommand{\ds}{\displaystyle}

\begin{document}

\bibliographystyle{amsplain}

\title[Radius of convexity of partial sums of odd functions]%
{Radius of convexity of partial sums of odd functions in the close-to-convex family}

\def\thefootnote{}
\footnotetext{ \texttt{\tiny File:~\jobname .tex,
          printed: \number\day-\number\month-\number\year,
          \thehours.\ifnum\theminutes<10{0}\fi\theminutes}
} \makeatletter\def\thefootnote{\@arabic\c@footnote}\makeatother

\author{Sarita Agrawal}
\address{Sarita Agrawal, Discipline of Mathematics,
Indian Institute of Technology Indore,
Simrol, Khandwa Road, Indore 452 020, India}
\email{saritamath44@gmail.com}

\author{Swadesh Kumar Sahoo${}^*$}
\address{Swadesh Kumar Sahoo, Discipline of Mathematics,
Indian Institute of Technology Indore,
Simrol, Khandwa Road, Indore 452 020, India}
\email{swadesh@iiti.ac.in}

\thanks{${}^*$ The corresponding author}
\begin{abstract}
We consider the class of all analytic and locally univalent functions $f$ of the form 
$f(z)=z+\sum_{n=2}^\infty a_{2n-1} z^{2n-1}$, $|z|<1$, satisfying the condition 
$$ \real\left(1+\frac{zf^{\prime\prime}(z)}{f^\prime (z)}\right)>-\frac{1}{2}.
$$
We show that every section $s_{2n-1}(z)=z+\sum_{k=2}^na_{2k-1}z^{2k-1}$, of $f$, is convex in the disk $|z|<\sqrt{2}/3$. 
We also prove that the radius $\sqrt{2}/3$ is best possible, i.e. the number $\sqrt{2}/3$ cannot be replaced by a larger one.
\\ 
\smallskip
\noindent
{\bf 2010 Mathematics Subject Classification}. Primary: 30C45; Secondary: 30C55.

\smallskip
\noindent
{\bf Key words and phrases.} 
Analytic, univalent, convex, and close-to-convex functions, odd functions, radius of convexity, partial sums or sections. 
\end{abstract}

\maketitle
\pagestyle{myheadings}
\markboth{S. Agrawal and S. K. Sahoo}{Radius of convexity of partial sums of odd functions}

\section{Introduction and Main Result}
Let $\mathcal{A}$ denote the class of all normalized analytic functions $f$ in the open unit disk 
$\D:=\{z\in\mathbb{C}:\,|z|<1\}$, i.e. $f$ has the Taylor series expansion
\begin{equation}\label{sec1-eqn1}
f(z)=z+\sum_{n=2}^\infty a_n z^n.
\end{equation}
The Taylor polynomial $s_n(z)=s_n(f)(z)$ of $f$ in $\mathcal{A}$, defined by,
$$s_n(z)=z+\sum_{k=2}^n a_k z^k
$$
is called the $n$-th {\em section/partial sum} of $f$.
Denote by $\mathcal{S}$, the class of {\em univalent} functions in $\mathcal{A}$.
A function $f\in\mathcal{A}$ is said to be {\em locally univalent}
at a point $z_0\in D\subset \mathbb{C}$ if it is univalent in some neighborhood of $z_0$;
equivalently $f'(z_0)\neq 0$. A function $f\in \mathcal{A}$ is called {\em convex} if $f(\mathbb{D})$
is a convex domain. The set of all convex functions are denoted by $\mathcal{C}$. The functions $f\in\mathcal{C}$
are characterized by the well-known fact
$${\rm Re}\left(1+\frac{zf''(z)}{f'(z)}\right)>0,\quad |z|<1.
$$
In this article, we mainly focus on a class, denoted by $\mathcal{L}$, of all 
locally univalent {\em odd} functions $f$ satisfying
\begin{equation}\label{sec1-eqn2}
\real\left(1+\frac{zf''(z)}{f'(z)}\right)>-\frac{1}{2}, \quad z\in \D.
\end{equation} 
Clearly, a function $f\in\mathcal{L}$ will have the Taylor series expansion 
$f(z)=z+\sum_{n=2}^{\infty}{a_{2n-1} z^{2n-1}}$. 
The function $f_0(z)=z/\sqrt{1-z^2}$ plays the role of an extremal function 
for $\mathcal{L}$; see for instance \cite[p.~68, Theorem~2.6i]{MM00}. 
This article is devoted to finding the largest disk $|z|<r$ in which every section 
$s_{2n-1}(z)=z+\sum_{k=2}^na_{2k-1}z^{2k-1}$, of $f\in \mathcal{L}$, is convex;
that is, $s_{2n-1}$ satisfies
$${\rm Re}\left(1+\frac{zs_{2n-1}''(z)}{s_{2n-1}'(z)}\right)>0. 
$$
Our main objective in this article is to prove

\medskip
\noindent
{\bf Main Theorem.} {\em Every section of a function in $\mathcal{L}$ is convex in the disk $|z|< \sqrt{2}/3$. 
The radius $\sqrt{2}/3$ cannot be replaced by a greater one.}

\medskip
\noindent
This observation is also explained geometrically in Figure~\ref{fig1}
by considering the third partial sum, $s_{3,0}$, of the extremal function
$f_0$. We next discuss some motivational background of our problem.

\begin{figure}[H]\label{fig1}
\begin{minipage}[b]{0.5\textwidth}
\includegraphics[width=6.7cm]{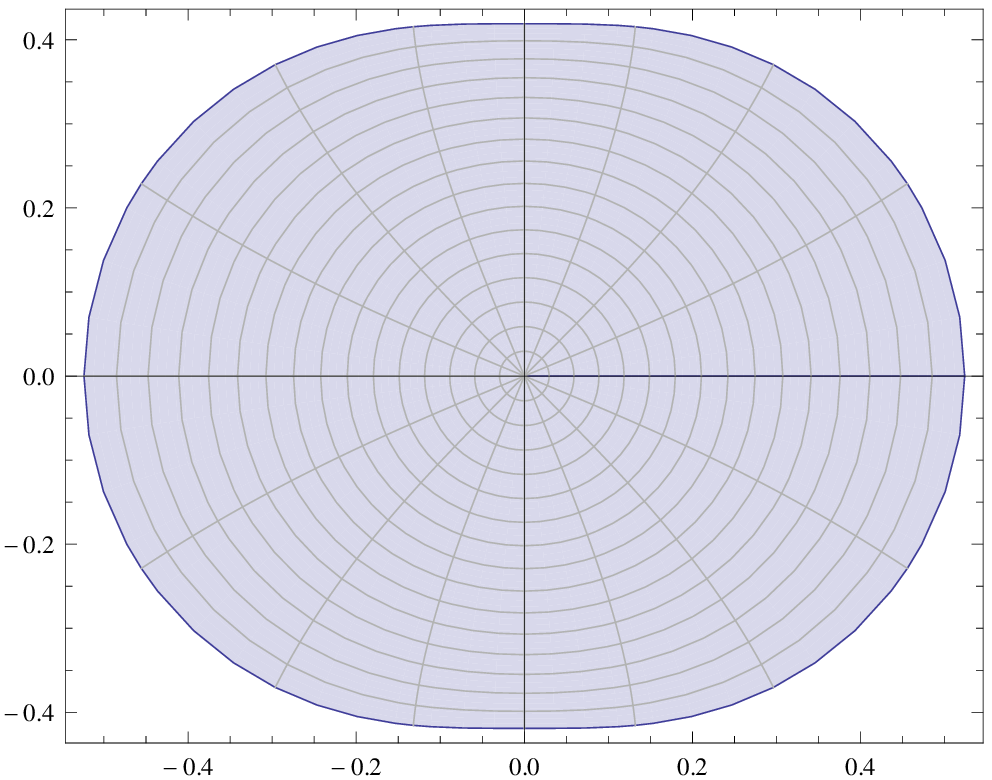}
\end{minipage}
\begin{minipage}[b]{0.45\textwidth}
\hspace*{-0.3cm}\includegraphics[height=5.3cm,width=6.7cm]{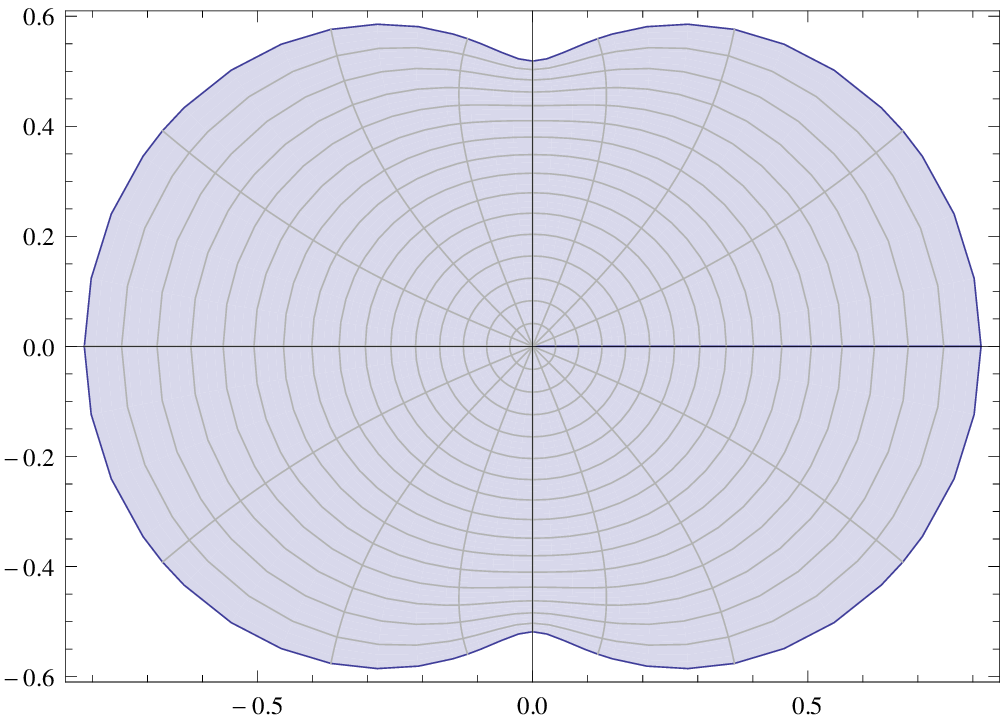}
\end{minipage}
\caption{The first figure shows convexity of the image domain $s_{3,0}(z)$ 
for $|z|<\sqrt{2}/3$ and the second figure 
shows non-convexity of the image domain $s_{3,0}(z)$ for $|z|<2/3=:r_0~(r_0>\sqrt{2}/3)$.}
\end{figure}

Considering odd univalent functions and studying classical problems 
of univalent function theory such as (successive) coefficient bounds,
inverse functions, etc. are quite interesting and found throughout the  
literature; see for instance \cite{Ke86,LZ84,Mil81,Ye05}.
In fact, an application of the Cauchy-Schwarz inequality shows that
the conjecture of Robertson: $1+|c_3|^2+|c_5|^2+\cdots+|c_{2n-1}|^2\le n$,
$n\ge 2$, for each odd function $f(z)=z+c_3z^3+c_5z^5+\cdots$ of $\mathcal{S}$,
stated in 1936 implies the well-known Bieberbach conjecture \cite{Rob36};
see also \cite{Dur77}. 
In our knowledge, studying radius properties for sections of odd univalent functions 
are new (as we do not find in the literature). 

Note that a subclass denoted by $\mathcal{F}$, of the class, $\mathcal{K}$, 
of close-to-convex functions, consisting of all locally univalent 
functions $f\in\mathcal{A}$ satisfying 
the condition (\ref{sec1-eqn2}) was considered in \cite{PSY14}. In this paper,
we consider functions from $\mathcal{F}$ that have odd Taylor coefficients.
Note that the following inclusion relations hold:
$$\mathcal{L}\subsetneq \mathcal{F}\subsetneq \mathcal{K}\subsetneq \mathcal{S}.
$$
The fact that functions in $\mathcal{F}$ are close-to-convex may be obtained
as a consequence of the result due to Kaplan (see \cite[p. 48, Theorem 2.18]{Dur83}).
In \cite{PSY14}, Ponnusamy et. al. have shown that every section of a function in the class $\mathcal{F}$ is convex in the disk $|z|<1/6$ and the radius $1/6$ is the best possible. 
They conjectured that every section of 
functions in the family $\mathcal{F}$ is univalent and close-to-convex in the disk $|z|<1/3$.
This conjecture has been recently settled by Bharanedhar and Ponnusamy in \cite[Theorem~1]{BP}.

The problem of finding the radius of univalence of sections of $f$ in $\mathcal{S}$ 
was first initiated by Szeg\"o in 1928. 
According to the Szeg\"o theorem 
\cite[Section~8.2, p. 243-246]{Dur83}, every section $s_n(z)$ of a function 
$f\in \mathcal{S}$ is univalent in the disk $|z|<1/4$; see \cite{Sze28} for the original paper.
The radius $1/4$ is best
possible and can be verified from the second partial sum of the Koebe function
$k(z)=z/(1-z)^2$.
Determining the exact (largest) radius of univalence
$r_n$ of $s_n(z)$ ($f\in\mathcal{S}$) remains an open problem.
However, many other related problems 
on sections have been solved for various geometric subclasses of $\mathcal{S}$,
eg. the classes $\mathcal{S}^*$, $\mathcal{C}$ and $\mathcal{K}$ of starlike, convex and
close-to-convex functions, respectively (see Duren \cite[\S8.2, p.241--246]{Dur83},
\cite{Goo83,Rob41,Rus72,Sma70} and the survey articles \cite{Ili79,Rav12}).
In \cite{Mac62}, MacGregor considered the class 
$$\mathcal{R}=\{f\in \mathcal{A}:\real f'(z)>0, z\in \D\}
$$
and proved that the partial sums $s_n(z)$ of $f\in \mathcal{R}$ are univalent 
in $|z|<1/2$, where the radius $1/2$ is best possible. On the other hand, in \cite{Sin70},
Ram Singh obtained the best radius, $r=1/4$, of convexity for sections of functions in the class 
$\mathcal{R}$. The reader can refer to \cite{PP74} for related information. Radius of
close-to-convexity of sections of close-to-convex functions is obtained in
\cite{Miki56}.

By the argument principle, it is clear that the $n$-th section $s_n(z)$ of an arbitrary
function in $\mathcal{S}$ is univalent in each fixed compact subdisk 
$\overline{\D_r}:=\{z\in \D:|z|\le r\}(r<1)$ of $\D$ provided that $n$ is sufficiently
large. In this way one can get univalent polynomials in $\mathcal{S}$ by setting 
$p_n(z)= \frac{1}{r}s_n(rz)$. Consequently, the set of all univalent polynomials 
is dense in the topology of locally uniformly convergence in 
$\mathcal{S}$. The radius of starlikeness of the partial sums $s_n(z)$ of $f\in\mathcal{S}^*$
was obtained by Robertson in \cite{Rob41}; (see also \cite[Theorem~2]{Sil88}) in the following form:

\medskip
\noindent
{\bf Theorem~A.} \cite{Rob41}
{\em If $f\in \mathcal{S}$ is either starlike, convex, typically-real, or convex in the direction of imaginary axis, then there
is an $N$ such that, for $n\ge N$, the partial sum $s_n(z)$ has the same property 
in $\D_r:=\{z\in \D:|z|<r\}$, where $r\ge 1-3(\log n)/n$.}

\medskip
\noindent
However, Ruscheweyh in \cite{Rus88} proved a stronger result by showing that the partial
sums $s_n(z)$ of $f$ are indeed starlike in $\D_{1/4}$ for functions $f$ belonging not only 
to $\mathcal{S}$ but also to the closed convex hull of $\mathcal{S}$. 
Robertson \cite{Rob41} further showed that sections of the Koebe function $k(z)$ are univalent
in the disk $|z|<1-3 n^{-1} \log n$ for $n\ge 5$,
and that the constant $3$ cannot be replaced by a smaller constant. However, Bshouty
and Hengartner \cite{BH91} pointed out that the Koebe function is not extremal for
the radius of univalency of the partial sums of $f\in \mathcal{S}$. A well-known
theorem by Ruscheweyh and Sheil-Small \cite{RS73} on convolution allows us to conclude immediately 
that if $f$ belongs to $\mathcal{C},\mathcal{S}^*,$ or $\mathcal{K}$, then its $n$-th section is 
respectively convex, starlike, or close-to-convex in the disk 
$|z|<1-3n^{-1} \log n$, for $n\ge 5$.
Silverman in \cite{Sil88} proved that the radius of starlikeness for sections of 
functions in the convex family $\mathcal{C}$ is $(1/2n)^{1/n}$ for all $n$. 
We suggest readers refer to \cite{PSY14,Rus72,Sma70,Sze28} and recent articles \cite{OP11,OP13,OP14,OPW13} for further interest on this topic. It is worth recalling that
radius properties of harmonic sections have recently been studied in \cite{KPV14,LS13-1,LS13-2,LS15,PS15}. 

\section{Preparatory results}
In this section we derive some useful results to prove our main theorem.
\begin{lemma}\label{sec2-lem1}
If $f(z)=z+\sum_{n=2}^{\infty}{a_{2n-1} z^{2n-1}}\in \mathcal{L}$, then the following estimates are obtained:
\begin{itemize}
\item[\bf (a)] $|a_{2n-1}|\leq \frac{(2n-2)!}{2^{2n-2}(n-1)!^2}$ for $n\ge2$. The equality holds for 
$$f_0(z)=\frac{z}{\sqrt{1-z^2}}
$$
or its rotation.

\medskip
\item[\bf (b)] $\left|\frac{zf^{\prime\prime}(z)}{f^{\prime}(z)}\right|\leq \frac{3r^2}{1-r^2}$ for $|z|=r<1$. The inequality is sharp.

\medskip
\item[\bf (c)] $\frac{1}{(1+r^2)^{3/2}}\leq |f^{\prime}(z)|\leq \frac{1}{(1-r^2)^{3/2}}$ for $|z|=r<1$. The inequality is sharp.

\medskip
\item[\bf (d)] If $f(z)=s_{2n-1}(z)+\sigma_{2n-1}(z)$, with $\sigma_{2n-1}(z)=\sum_{k=n+1}^{\infty}{a_{2k-1} z^{2k-1}}$, 
then for $|z|=r<1$ we have
$$|\sigma_{2n-1}^{\prime}(z)|\leq A(n,r)
~~\mbox{ and }~~
|z\sigma_{2n-1}^{\prime\prime}(z)|\leq B(n,r),
$$
where
$$A(n,r)=\sum_{k=n+1}^{\infty}\frac{(2k-1)!}{2^{2k-2}(k-1)!^2}r^{2k-2}
~~\mbox{ and }~~
B(n,r)=\sum_{k=n+1}^{\infty}\frac{(2k-2)(2k-1)!}{2^{2k-2}(k-1)!^2}r^{2k-2}.
$$
The ratio test guarantees that both the series are convergent.
\end{itemize}
\end{lemma}

\begin{proof}
{\bf (a)} Set 
\begin{equation}\label{Neweqn}
p(z)=1+\ds\frac{2}{3}\left(\frac{zf^{\prime\prime}(z)}{f^{\prime}(z)}\right).
\end{equation}
Clearly, $p(z)=1+\sum_{n=1}^{\infty}{p_n z^n}$ is analytic in $\D$ and $\real p(z)>0$ there. 
So, by Carath{\'e}odory Lemma, we obtain that $|p_n|\leq 2$ for all $n \ge 1$. 
Putting the series expansions for $f'(z),~f''(z)$ and $p(z)$ in (\ref{Neweqn}) we get
$$\sum_{n=2}^{\infty}{(2n-1)(2n-2)a_{2n-1}z^{2n-1}}
=\frac{3}{2}\sum_{n=2}^{\infty}\left(\sum_{k=1}^{n-1} p_{2k-1}(2n-2k-1)a_{2n-2k-1}\right)z^{2n-2}
$$
$$+\frac{3}{2}\sum_{n=2}^{\infty}\left(\sum_{k=1}^{n-1} p_{2k}(2n-2k-1)a_{2n-2k-1}\right)z^{2n-1}.
$$
Equating the coefficients of $z^{2n-1}$ and $z^{2n-2}$ on both sides, we obtain
$$\sum_{k=1}^{n-1} p_{2k-1}(2n-2k-1)a_{2n-2k-1}=0
$$
and
\begin{equation}\label{sec2-eqn0}
(2n-1)(2n-2)a_{2n-1}=\frac{3}{2}\sum_{k=1}^{n-1} p_{2k}(2n-2k-1)a_{2n-2k-1},
\quad \mbox{ for all } n\ge 2.
\end{equation}
Hence,
\begin{equation}\label{sec2-eq1}
|a_{2n-1}|\leq\frac{3}{(2n-1)(2n-2)}\sum_{k=1}^{n-1}(2k-1)|a_{2k-1}|.
\end{equation}
For $n=2$, we can easily see that $|a_3|\leq 1/2$, and for $n=3$, we have
$$|a_5|\leq \frac{3}{20}(1+3|a_3|)\leq \frac{3}{8}.
$$
Now, we can complete the proof by method of induction. Therefore, if we assume $|a_{2k-1}|\leq \frac{(2k-2)!}{2^{2k-2}(k-1)!^2}$ for $k=2, 3, \ldots , n-1$, then we deduce 
from (\ref{sec2-eq1}) that
$$|a_{2n-1}|\leq\frac{3}{(2n-1)(2n-2)}\sum_{k=1}^{n-1} {\frac{(2k-1)!}{2^{2k-2}(k-1)!^2}}.
$$
Induction principle tells us to show that 
$$|a_{2n-1}|\leq \frac{(2n-2)!}{2^{2n-2}(n-1)!^2}.
$$
It suffices to show that
$$\frac{3}{(2n-1)(2n-2)}\sum_{k=1}^{n-1} {\frac{(2k-1)!}{2^{2k-2}(k-1)!^2}}=\frac{(2n-2)!}{2^{2n-2}(n-1)!^2}
$$
or,
$$\sum_{k=1}^{n-1} {\frac{3(2k-1)!}{2^{2k-2}(k-1)!^2}}=\frac{(2n-2)(2n-1)!}{2^{2n-2}(n-1)!^2}.
$$
Again, we prove this by method of induction. It can easily be seen that for $k=1$ it is true. 
Assume that it is true for $k=2, 3, \ldots , n-1$, then we have to prove that
$$\sum_{k=1}^{n} {\frac{3(2k-1)!}{2^{2k-2}(k-1)!^2}}=\frac{(2n)(2n+1)!}{2^{2n}(n)!^2},
$$
which is easy to see, since
$$\sum_{k=1}^{n} {\frac{3(2k-1)!}{2^{2k-2}(k-1)!^2}}=\frac{(2n-2)(2n-1)!}{2^{2n-2}(n-1)!^2}+\frac{3(2n-1)!}{2^{2n-2}(n-1)!^2}=\frac{(2n)(2n+1)!}{2^{2n}(n)!^2}.
$$
Hence, the proof is complete. For equality, it can easily be seen that
$$ f_0(z)=\frac{z}{\sqrt{1-z^2}}=z+\sum_{n=2}^{\infty} \frac{(2n-2)!}{2^{2n-2}(n-1)!^2}
z^{2n-1}
$$
belongs to $\mathcal{L}$. 

The image of the unit disk $\D$ under $f_0$ is
shown in Figure~\ref{sec2-fig1} which indicates that $f_0(\D)$ is not convex.

\begin{figure}[H]
\includegraphics[height=10cm,width=7cm]{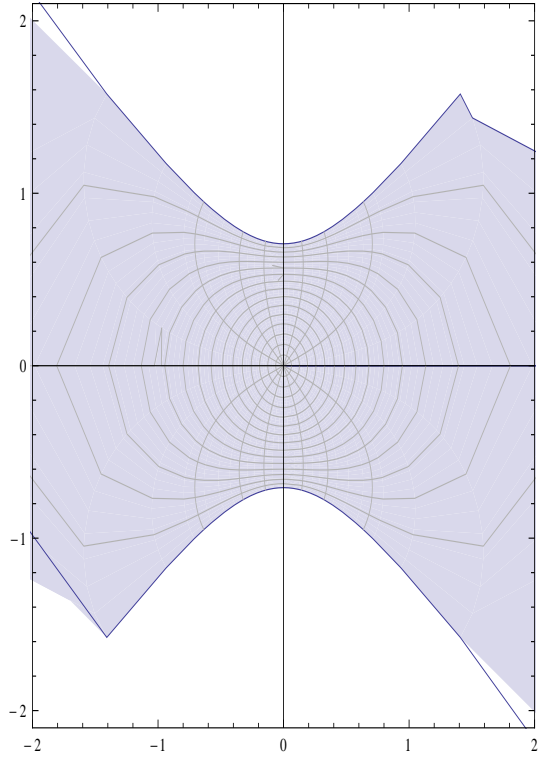}
\caption{The image domain $f_0(\D)$, where $f_0(z)=\frac{z}{\sqrt{1-z^2}}$.}\label{sec2-fig1}
\end{figure}

{\bf (b)} We see from the definition of $\mathcal{L}$ that
$$1+\frac{zf^{\prime\prime}(z)}{f^\prime (z)}\prec \frac{1+2z^2}{1-z^2},\quad 
\mbox{i.e.}, \frac{zf^{\prime\prime}(z)}{f^\prime (z)}\prec \frac{3z^2}{1-z^2}=:h(z),
$$
where $\prec$ denotes the usual subordination. The poof of (b) now follows easily.

{\bf (c)} Since 
$$ \frac{zf^{\prime\prime}(z)}{f^\prime (z)}\prec h(z),
$$
it follows by the well-known subordination result due to Suffridge \cite{Suf70} that
$$f^{\prime}(z) \prec \exp\left(\int_0 ^ z {\frac{h(t)}{t}\mbox{d}t}\right)
=\exp\left(3\int_0^z {\frac{t}{1-t^2}\mbox{d}t}\right)=\frac{1}{(1-z^2)^{3/2}}.
$$
Hence, the proof of (c) follows.

{\bf (d)} By $(a)$, we see that
$$|\sigma_{2n-1}^{\prime}(z)|\leq \sum_{k=n+1}^{\infty}(2k-1)|a_{2k-1}|r^{2k-2}
\leq A(n,r).
$$
and
$$|z\sigma_{2n-1}^{\prime\prime}(z)|\leq \sum_{k=n+1}^{\infty}(2k-1)(2k-2)|a_{2k-1}|r^{2k-2}\leq B(n,r).
$$
The proof of our lemma is complete.
\end{proof}

\section{Proof of the Main Theorem}
For an arbitrary $f(z)=z+\sum_{n=2}^{\infty}{a_{2n-1} z^{2n-1}}\in \mathcal{L}$, we first consider its third section $s_3(z)=z+a_3z^3$ of $f$. Simple computation shows
$$1+\frac{zs_3 ^{\prime\prime}(z)}{s_3^{\prime}(z)}=1+\frac{6a_3z^2}{1+3a_3z^2}.
$$
By using Lemma~\ref{sec2-lem1}(a), we have $|a_3|\le 1/2$ and hence
$$\real\left(1+\frac{zs_3 ^{\prime\prime}(z)}{s_3^{\prime}(z)}\right)\ge 1-\frac{6|a_3||z|^2}{1-3|a_3||z|^2} \ge 1-\frac{3|z|^2}{1-\frac{3}{2}|z|^2}
$$
which is positive for $|z|<\sqrt{2}/3$. Thus, $s_3(z)$ is convex in the disk $|z|<\sqrt{2}/3$. To show that the constant $\sqrt{2}/3$ is best possible, we consider the function $f_0(z)$ defined by
$$f_0(z)=\frac{z}{\sqrt{1-z^2}}.
$$
We denote by $s_{3,0}(z)$, the third partial sum $s_3(f_0)(z)$ of $f_0(z)$ so that $s_{3,0}(z)=z+(1/2)z^3$ and hence, we find 
$$1+\frac{zs_{3,0} ^{\prime\prime}(z)}{s_{3,0}^{\prime}(z)}=\frac{2+9z^2}{2+3z^2}.
$$
This shows that
$$\real\left(1+\frac{zs_{3,0} ^{\prime\prime}(z)}{s_{3,0}^{\prime}(z)}\right)=0
$$
when $z^2=(-2/9) \mbox{ or } (-2/3)$ \quad i.e., when $|z|^2=(2/9) \mbox{ or } (2/3)$.
Hence, the equality occurs.

Next, let us consider the case $n=3$.
Our aim in this case is to show that
$$\real\left(1+\frac{zs_5^{\prime\prime}(z)}{s_5^{\prime}(z)}\right)=\real\left(\frac{1+9a_3 z^2+25 a_5 z^4}{1+3a_3 z^2+5a_5 z^4}\right)>0
$$
for $|z|<\sqrt{2}/3$. Since the real part $\real[(1+9a_3 z^2+25 a_5 z^4)/(1+3a_3 z^2+5a_5 z^4)]$ is harmonic in $|z|\leq \sqrt{2}/3$, it suffices to check that
$$\real\left(\frac{1+9a_3 z^2+25 a_5 z^4}{1+3a_3 z^2+5a_5 z^4}\right)>0
$$
for $|z|=\sqrt{2}/3$. Also we see that
$$\real\left(\frac{1+9a_3 z^2+25 a_5 z^4}{1+3a_3 z^2+5a_5 z^4}\right)=3-\real\left(\frac{2-10 a_5 z^4}{1+3a_3 z^2+5a_5 z^4}\right)\ge 3-\left|\frac{2-10 a_5 z^4}{1+3a_3 z^2+5a_5 z^4}\right|
$$
and, so by considering a suitable rotation of $f(z)$, the proof reduces to $z=\sqrt{2}/3$; this means that it is enough to prove
$$\frac{3}{2}> \left|\frac{81-20 a_5}{81+54a_3 +20a_5}\right|.
$$
From (\ref{sec2-eqn0}), we have
$$ a_3=\frac{p_2}{4} \quad \mbox{and}\quad a_5=\left(\frac{3}{40}\right)\left(\frac{3}{4} p_2^2+p_4\right).
$$
Since $|p_2|\le 2$ and $|p_4|\le 2$, it is convenient to rewrite the last two relations as
$$a_3=\frac{\alpha}{2} \quad \mbox{and} \quad a_5=\frac{3}{40}(3\alpha^2+2\beta)
$$
for some $|\alpha|\le 1$ and $|\beta|\le 1$.

Substituting the values for $a_3$ and $a_5$, and applying the maximum principle in 
the last inequality, it suffices to show the inequality
$$\frac{3}{2}\left|81+27\alpha+\frac{9\alpha^2}{2}+3\beta\right|> \left|81-\frac{9\alpha^2}{2}-3\beta\right|
$$
for $|\alpha|=1=|\beta|$. Finally, by the triangle inequality, the last inequality follows if we can show that
$$9\left|9+3\alpha+\frac{\alpha^2}{2}\right|-6\left|9-\frac{\alpha^2}{2}\right|>5
$$
which is easily seen to be equivalent to
$$9\left|9\overline{\alpha}+3+\frac{\alpha}{2}\right|-6\left|9\overline{\alpha}-\frac{\alpha}{2}\right|>5
$$
as $|\alpha|=1$. Write $\real (\alpha)=x$. It remains to show that
$$T(x):=9\sqrt{18x^2+57x+\frac{325}{4}}-6\sqrt{\frac{361}{4}-18x^2}> 5
$$
for $-1\leq x\leq1$.
\begin{figure}[H]
\includegraphics{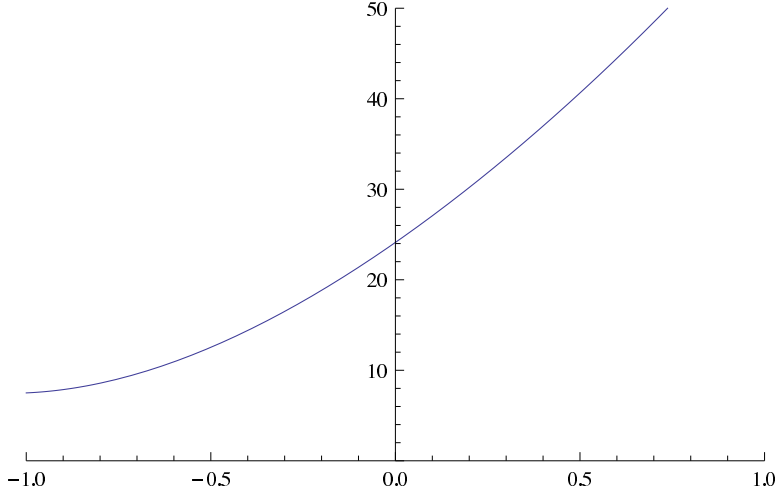}
\caption{Graph of $T(x)$.}
\end{figure}
It suffices to show
$$9\sqrt{18x^2+57x+\frac{325}{4}}> 5+6\sqrt{\frac{361}{4}-18x^2}.
$$
Squaring both sides we have
$$ 2106 x^2+4617x+\frac{13229}{4} > 60\left(\sqrt{\frac{361}{4}-18x^2}\right).
$$
Again by squaring both sides we have
$$ \left(2106 x^2+4617x+\frac{13229}{4}\right)^2 > 3600\left(\frac{361}{4}-18x^2\right).
$$
After computing, it remains to show that $\phi(x)>0$, where 
$$\phi(x)=ax^4+bx^3+cx^2+dx+e
$$
and the coefficients are
$$a=4435236, b=19446804, c=35311626, d=30539146.5, e=10613002.5625. 
$$
Here we see that $\phi^{iv}(x)=24a>0$. Thus the function $\phi^{\prime\prime\prime}(x)$
is increasing in $-1\le x\le 1$ and hence $\phi^{\prime\prime\prime}(x)\ge \phi^{\prime\prime\prime}(-1)=10235160>0$. This implies $\phi^{\prime\prime}(x)$ is increasing. Hence $\phi^{\prime\prime}(x)\ge \phi^{\prime\prime}(-1)=7165260>0$.
Consequently, $\phi^{\prime}(x)$ is increasing and we have $\phi^{\prime}(x)\ge \phi^{\prime}(-1)=515362.5>0$. Finally we get, $\phi(x)$ is increasing and hence we have $\phi(x)>\phi(-1)=373914.0625>0$.
This completes the proof for $n=3$.

We next consider the general case $n\ge 4$. It suffices to show that 
$$\real\left(1+\frac{zs_{2n-1}^{\prime\prime}}{s_{2n-1}^\prime}\right)>0 \quad \mbox{for} \quad |z|=r
$$
with $r=\sqrt{2}/3$ for all $n\ge 4$. From the maximum modulus principle, we shall then conclude that the last inequality holds for all $n\ge 4$
$$\real\left(1+\frac{zs_{2n-1}^{\prime\prime}}{s_{2n-1}^\prime}\right)>0
$$
for $|z|<\sqrt{2}/3$. In other words, it remains to find the largest $r$ so that 
the last inequality holds for all $n\ge 4$. 

By the same setting of $f(z)$ as in Lemma~\ref{sec2-lem1}(d),
it follows easily that
$$1+\frac{zs_{2n-1}^{\prime\prime}}{s_{2n-1}^\prime}=1+\frac{z(f^{\prime\prime}(z)-\sigma_{2n-1}^{\prime\prime}(z))}{f^{\prime}(z)-\sigma_{2n-1}^{\prime}(z)}=1+\frac{zf^{\prime\prime}(z)}{f^{\prime}(z)}+\frac{\frac{zf^{\prime\prime}(z)}{f^{\prime}(z)}\sigma_{2n-1}^{\prime}(z)-z\sigma_{2n-1}^{\prime\prime}(z)}{f^{\prime}(z)-\sigma_{2n-1}^{\prime}(z)}
$$
or, 
$$\real\left(1+\frac{zs_{2n-1}^{\prime\prime}}{s_{2n-1}^\prime}\right)\ge 1-\left|\frac{zf^{\prime\prime}(z)}{f^{\prime}(z)}\right|-\frac{\left|\frac{zf^{\prime\prime}(z)}{f^{\prime}(z)}\right||\sigma_{2n-1}^{\prime}(z)|+|z\sigma_{2n-1}^{\prime\prime}(z)|}{|f^{\prime}(z)|-|\sigma_{2n-1}^{\prime}(z)|}.
$$
Then by using Lemma~\ref{sec2-lem1}, we obtain
$$\real\left(1+\frac{zs_{2n-1}^{\prime\prime}}{s_{2n-1}^\prime}\right)\ge 1-\frac{3r^2}{1-r^2}-\frac{\left(\frac{3r^2}{1-r^2}\right)A(n,r)+B(n,r)}{\frac{1}{(1+r^2)^{(3/2)}}-A(n,r)}.
$$
Thus, we conclude that
$$\real\left(1+\frac{zs_{2n-1}^{\prime\prime}}{s_{2n-1}^\prime}\right)>0
$$
provided
$$\frac{1-4r^2}{1-r^2}-\frac{(1+r^2)^{3/2}}{1-r^2}\left(\frac{3r^2 A(n,r)+(1-r^2)B(n,r)}{1-(1+r^2)^{3/2}A(n,r)}\right)>0,
$$
or, equivalently
$$(1+r^2)^{3/2}\left(\frac{3r^2 A(n,r)+(1-r^2)B(n,r)}{1-(1+r^2)^{3/2}A(n,r)}\right)<1-4r^2.
$$
We show that the above relation holds for all $n\ge 4$ with $r=\sqrt{2}/3$.
The choice $r=\sqrt{2}/3$ brings the last inequality to the form
$$\left(\frac{11}{9}\right)^{3/2}\left(\frac{\frac{2}{3} A(n,\frac{\sqrt{2}}{3})+\frac{7}{9}B(n,
\frac{\sqrt{2}}{3})}{1-(\frac{11}{9})^{3/2}A(n,\frac{\sqrt{2}}{3})}\right)< \frac{1}{9}.
$$
Set
$$ C\left(n,\frac{\sqrt{2}}{3}\right):=1-\left(\frac{11}{9}\right)^{3/2}A\left(n,\frac{\sqrt{2}}{3}\right).
$$
We shall prove that $C\left(n,\frac{\sqrt{2}}{3}\right)>0$ for $n\ge 4$ i.e., $$A\left(n,\frac{\sqrt{2}}{3}\right)<\frac{27}{(11)^{3/2}}$$
and
$$A\left(n,\frac{\sqrt{2}}{3}\right)+B\left(n,\frac{\sqrt{2}}{3}\right)<\frac{27}{7\times (11)^{3/2}}\quad \mbox{ for } n\ge 4.
$$
If the last inequality is proved, then automatically the previous one follows. 
Hence, it is enough to prove the last inequality. Now,
\begin{eqnarray*}
A(n,r)+B(n,r) &=& \sum_{k=n+1}^{\infty}\frac{(2k-1)(2k-1)!}{2^{2k-2}(k-1)!^2}(r^2)^{k-1}\\ &\le & \sum_{k=5}^{\infty}\frac{(2k-1)(2k-1)!}{2^{2k-2}(k-1)!^2}(r^2)^{k-1}\\ 
&= & \sum_{k=1}^{\infty}\frac{(2k-1)(2k-1)!}{2^{2k-2}(k-1)!^2}(r^2)^{k-1} 
- \sum_{k=1}^{4}\frac{(2k-1)(2k-1)!}{2^{2k-2}(k-1)!^2}(r^2)^{k-1}\\
&=& \frac{1+2 r^2}{(1-r^2)^{5/2}}-\left(1+\frac{9}{2}r^2+\frac{75}{8}r^4+
\frac{245}{16}r^6 \right).
\end{eqnarray*}
Substituting the value $r=\sqrt{2}/3$, we obtain 
$$A\left(n,\frac{\sqrt{2}}{3}\right)+B\left(n,\frac{\sqrt{2}}{3}\right) \le 0.076\cdots < 0.105\cdots = \frac{27}{7\times (11)^{3/2}}.
$$
This completes the proof of our main theorem.
\hfill{$\Box$}

\vskip 1cm
\noindent
{\bf Acknowledgement.} The work of the first author is supported by University
Grants Commission, New Delhi (grant no. F.2-39/2011 (SA-I)).

\end{document}